\documentclass[11pt,a4paper]{article}
\usepackage[OT1]{fontenc}
\usepackage[all]{xy}
\usepackage[english]{babel}
\usepackage{amsmath}
\usepackage{amsfonts}
\usepackage{amsthm}
\DeclareMathOperator{\sinc}{sinc}
\DeclareMathOperator{\rank}{rank}

\def\d{\displaystyle}
\def\a{ {\cal A} }
\def\h{ {\cal H} }
\def\ele{ {\cal L} }
\def\b{ {\cal B} }
\def\u{ {\cal U} }

\def\o{ {\cal O} }
\def\p{ {\cal P} }

\def\k{ {\cal K} }

\def\G{ {\cal G} }

\def\w{ {\cal W} }

\def\noi{\noindent}
\def\g1{ \mathfrak{g}_1  }

\newtheorem{teo}{Theorem}[section]
\newtheorem{prop}[teo]{Proposition}
\newtheorem{lem}[teo]{Lemma}
\newtheorem{coro}[teo]{Corollary}

\theoremstyle{definition}
\newtheorem{rem}[teo]{Remark}

\title{$p$-Schatten commutators of projections}
\author{E. Andruchow, M. E. Di Iorio y Lucero}

\begin{document}

\maketitle 

\begin{abstract}
Let $\h=\h_+\oplus\h_-$ be a fixed orthogonal decomposition of the complex Hilbert space $\h$ in two infinite dimensional subspaces. We study the geometry of the set $\p^p$ of selfadjoint projections in the Banach algebra
$$
\a^p=\{A\in\b(\h): [A,E_+]\in\b_p(\h)\},
$$
where $E_+$ is the projection onto $\h_+$ and $\b_p(\h)$ is the Schatten ideal of $p$-summable operators ($1\le p <\infty$). The norm in $\a^p$ is defined in terms of the norms of the matrix entries of the operators given by the above decomposition. The space $\p^p$ is shown to be a differentiable $C^\infty$ submanifold of $\a^p$, and a homogeneous space of the  group of unitary operators in $\a^p$. The connected components of $\p^p$ are characterized, by means of a partition of $\p^p$ in nine classes, four discrete classes and five essential classes:
\begin{itemize}
\item
the first two corresponding to finite rank or co-rank, with the connected components parametrized by theses ranks;
\item
the next two discrete classes carrying a Fredholm index, which parametrizes its components;
\item
the remaining essential classes, which are connected.
\end{itemize}
  \end{abstract}
\bigskip

{\bf 2010 MSC: 46H35, 47B10, 58B05}  

{\bf Keywords: Projections, Schatten $p$-ideals}   
\section{Introduction}
Let $\h=\h_+\oplus\h_-$ be an orthogonal decomposition of the complex Hilbert space $\h$ in two infinite dimensional closed subspaces, with corresponding projections $E_+$ and $E_-$. Consider the algebra
$$
\a^p:=\{A\in\b(\h): [A,E_+]=AE_+-E_+A \in\b_p(\h)\},
$$
where $\b(\h)$ denotes the algebra of bounded linear operators in $\h$, and $\b_p(\h)$ is the Schatten ideal of $p$-summable operators ($1\le p<\infty$). $\a^p$ is a $*$-Banach algebra with a suitable norm ($*$ is the usual adjoint). The purpose of this paper is the study of the set $\p^p$ of selfadjoint projections in $\a^p$,
$$
\p^p=\{P\in\a^p: P^2=P^*=P\}.
$$
It is known \cite{zemanek}, \cite{cprbanach} that the set of idempotents  ($Q$ such that $Q^2=Q$) of a Banach algebra is a complemented submanifold of the algebra. Here we show that also the set of selfadjoint idempotents is a submanifold of the algebra, in the case of the algebra $\a^p$ (it can be proved to hold for arbitrary $*$-Banach algebras). 
The case $p=2$ was extensively treated in \cite{tumpach}.

We characterize the connected components of $\p^p$. First, we see  that $\p^p$ is partitioned in nine classes, four discrete classes $\mathbb{D}_j$, $1\le j\le 4$, and five essential classes $\mathbb{E}_j$, $1\le j \le 5$. The first two discrete classes correspond to the projections of finite rank or finite co-rank, and its connected components are characterized by these numbers. The next two discrete classes are more interesting, and correspond to the so called $p$-{\it restricted Grassmannian}, associated to $E_+$, and to $E_-$, respectively. Projections in a restricted Grassmannian carry an integer  Fredholm index, which in turn parametrizes the connected components of $\mathbb{D}_3$ and $\mathbb{D}_4$ (as with the former two, one passes from one class to the other with the symmetry $P\mapsto P^\perp=1-P$, and thus the geometric and topological of both pairs are similar). The remaining essential classes are shown to be connected.

Examples of discrete projections (in the $p=1$ restricted Grassmannian) of the decompostition $L^2(\mathbb{T})=H^2(\mathbb{D})\oplus H^2_-(\mathbb{D})$, where  $H^2(\mathbb{D})$ is  the Hardy space of the disk,   are the projections onto the the subspaces $\varphi H^2(\mathbb{D})$, for $\varphi$ a  smooth function of modulus one. The index given by (minus) the winding number of $\varphi$.

Examples of essential projections, again for $p=1$, are given for the decomposition $L^2(\mathbb{R}^n)=L^2(\Omega)\oplus L^2(\Omega^c)$, where $\Omega\subset\mathbb{R}^n$ is a measurable set  with positive finite measure. In this setting, the projection $FE_+F^{-1}$ ($F=$  Fourier-Plancherel transform), onto the space of functions in $L^2(\mathbb{R}^n)$ with Fourier transform supported in $\Omega$, is an essential projection.

This study is a continuation of \cite{conmcomp}, where the ideal of compact operators was considered. Some of the techniques and results are similar in both contexts, the main difficulty in the case at hand ($p$-Schatten ideals) is that  the structure algebra $\a^p$ is a Banach algebra, whereas  in the compact case it is a C$^*$-algebra. For instance, we need to prove the smooth local structure of the group $\u_{\a^p}$ of unitary operators in $\a^p$, which acts in $\p^p$. 

We do not know if the geodesics of the Grassmann manifold of $\h$, lying in $\p^p$ (that is, with initial velocity in $\a^p$), are short for the Finsler metric given by the norm of the Banach algebra $\a^p$. However, we show that for the case of the discrete classes $\mathbb{D}_3$ and $\mathbb{D}_4$ (corresponding to the $p$-restricted Grassmannian), the connected component containing a given projection $P$, is a submanifold of the affine space $P+\b_p(\h)$, which carries naturally the Schatten $p$-norm. With the Finsler metric given by this norm, the geodesics of the full Grassmannian of $\h$, lying  inside this component, are short.

\section{Preliminaries}
For $1\le p <\infty$, we denote by $\b_p(\h)$ the ideal of $p$-Schatten operators in $\h$, i.e., $\b_p(\h)=\{T\in\b(\h): Tr(|T|^p)<\infty\}$, with its norm $\|T\|_p=Tr^{1/p}(|T|^p)$.  We fix an orthogonal decomposition 
$$
\h=\h_+\oplus\h_-
$$
with corresponding projections $E_+$ and $E_-$. We make the assumption that both $\h_+$ and $\h_-$ are infinite dimensional. Denote by $\p(\h)=\p$ the set of all orthogonal projections in $\h$.  We are interested in the set
$$
\p^p_{\h_+}=\p^p:=\{P\in\p: [P,E_+]\in\b_p(\h)\}.
$$
Accordingly, we denote by 
$$
\a^p_{\h_+}=\a^p:=\{A\in\b(\h): [A,E_+]\in\b_p(\h)\}.
$$
With $\a^p_h$ and $\a^p_{ah}$ we denote, respectively, the sets of selfadjoint and anti-Hermitian elements of $\a^p$.

In terms of the decomposition $\h=\h_+\oplus\h_-$, operators in $\h$ can be written as $2\times 2$ matrices. It is clear that elements of $\a^p$ are characterized as those matrices
$$
A=\left(\begin{array}{cc} A_{11} & A_{12} \\ A_{21} & A_{22} \end{array} \right)
$$ 
such that $A_{12}\in\b_p(\h_-,\h_+)$ and $A_{21}\in\b_p(\h_+,\h_-)$.

We endow $\a^p$ with the following norm 
\begin{equation}\label{norma}
\| A\|_{\infty,p}:=\|A_{11}\|+\|A_{22}\|+\|A_{12}\|_p+\|A_{21}\|_p.
\end{equation}

It is easy to see that $\a_p$ is a Banach space with this norm. Also it is clear that it is an algebra. This can be seen by elementary matrix computations, or noting that if $A, B\in\a^p$, then 
$$
[AB,E_+]=ABE_+-AE_+B+AE_+B-E_+AB=A[B,E_+]-[A,E_+]B\in\b_p(\h).
$$
With the norm $\| \, \, \|_{\infty,p}$ just defined, it is elementary that $\| AB\|_{\infty,p} \le \|  A\|_{\infty,p} \|  B\|_{\infty,p}$ for $A,B\in\a^p$, i.e., $\a^p$ is a Banach algebra. Note also that 
if $A\in\a^p$ then $A^*\in\a^p$ and $\|  A^*\|_{\infty,p}=\| A\|_{\infty,p}$. Is is also clear that $\|  \, \, \|_{\infty,p}$ is not a C$^*$-norm. However, the inclusion
$$
(\a^p,\|  \, \, \|_{\infty,p})\hookrightarrow (\b(\h),\|\, \|)
$$
is continuous, so that $\a^p$ is a Banach subalgebra of $\b(\h)$.

Let us denote by $\G_{\a^p}$ the group of invertible elements in $\a^p$. It is usually known in the literature \cite{segal} as one of the reduced groups (acting in the restricted Grassmannian, see for instance \cite{tumpach}). For instance, it is known that if $G\in\G_{\a^p}$ then  its diagonal entries (in the $\h=\h_+\oplus\h_-$ matrix) are $p$-Fredholm operators (i.e., operators which are invertible modulo the ideal $\b_p(\h)$). As such, to  $G\in\G_{\a^p}$ an index can be attached, namely, the Fredholm index of the $1,1$-entry. As a consequence  $\G_{\a^p}$ is disconnected, and its connected components are parametrized by this index.  

Let us denote by $\u_{\a^p}:=\{U\in \G_{\a^p}: U \hbox{ is unitary  in } \h\}$.

The following elementary property of $\G_{\a^p}$ shall be very useful, it states that the usual polar decomposition of (invertible) elements of $\a^p$ stays in $\a^p$

\begin{prop}\label{polar}
If $G\in\G_{\a^p}$ and $G=U|G|$ is the polar decomposition, then $|G|, U\in\G_{\a^p}$ (i.e., $U\in\u_{\a^p}$). Moreover, $G$ and $U$ share the same index.
\end{prop}
\begin{proof}
It suffices to show that $|G|\in\G_{\a^p}$. Clearly $G^*G\in\G_{\a^p}$ Denote by $\sigma_{\b(\h)}(G^*G)$ and $\sigma_{\a^p}(G^*G)$ the spectra of $G^*G$ in $\b(\h)$ and $\a^p$, respectively. Since $\a^p\subset\b(\h)$ is a Banach subalgebra, its follows that (see for instance \cite{rickart})
$$
\sigma_{\b(\h)}(G^*G)\subset \sigma_{\a^p}(G^*G) \ \   \hbox{ and } \ \ \partial \sigma_{\a^p}(G^*G)\subset \partial  \sigma_{\b(\h)}(G^*G).
$$
On the other hand, since $G^*G$ is positive and invertible, $\sigma_{\b(\h)}(G^*G)\subset (0,+\infty)$. Then, $\sigma_{\b(\h)}(G^*G)=\partial \sigma_{\b(\h)}(G^*G)$. Thus $\sigma_{\b(\h)}(G^*G)=\sigma_{\a^p}(G^*G)=\sigma\subset [\delta,+\infty)$ for some $\delta>0$. Denote by $\log(z)$ the usual complex log function (discontinuous in the negative real axis). Note that $\log(z)$ is analytic on an open neighbourhood of $\sigma$, and thus $\log(G^*G)$ is defined in $\a^p$ by means of the usual holomorphic functional calculus in Banach algebras. Let $C=exp(\frac12  \log(G^*G))\in\a^p$. Note that if one regards $G^*G$ as an element in $\b(\h)$, $C\in\b(\h)$ is the usual (positive) square root of $G^*G$, i.e., $C=(G^*G)^{1/2}=|G|$.

The set of positive elements in $\G_{\a^p}$ is convex, and therefore connected. Thus $G$ and $U$ lie in the same connected component of $\G_{\a^p}$.
\end{proof}

\section{Regular structure of $\p^p$}
In this section we show that the set $\p^p$ of orthogonal projections in $\a^p$ is a   complemented $C^\infty$ submanifold of $\a^p_h$. It is known that the set of idempotents of a Banach algebra is a complemented submanifold of the algebra (see \cite{cprbanach} or \cite{zemanek}). Here we are dealing with selfadjoint idempotents. 

First, note that the group $\u_{\a^p}$ is a Banach-Lie group, whose Banach-Lie algebra is $\a^p_{ah}$: 
\begin{teo}
The group $\u_{\a^p}$ is a  Banach-Lie group and a complemented submanifold of $\a^p$. Its Banach-Lie algebra is $\a^p_{ah}$.
\end{teo}
\begin{proof}
The exponential map $\exp:\a^p\to\G_{\a^p}$, $exp(X)=e^X$, is a local diffeomorphism, there exists a  radius $0<r<1$ such and an open subset $0\in\w\subset\a^p$ such that 
$$
exp:\w\to\{G\in\a^p:\|G-1\|_{\infty,p}<r\}
$$
is a diffeomorphism. Its inverse is the usual $\log$ series. When restricted to $\a^p_{ah}$, it takes values in $\u_{\a^p}$, which is a complemented (real) subspace of $\a^p$. Thus, in order to obtain a local chart for $\u_{\a^p}$ around $1$ it suffices to show that elements in $\u_{\a^p}$ close enough to $1$, are of the form $e^X$ for some $X\in\a^p_{ah}$ close to $0$. In fact, if $U\in\u_{\a^p}$ satisfies  $\|U-1\|_{\infty,p}<r (<2)$,  since 
$$
\|U-1\|\le \|U-1\|_{\infty,p}<2,
$$
the spectrum $\sigma_{\b(\h)}(U)$ is contained in an arc $\{e^{i\theta}: |\theta|\le \theta_0<\pi\}$. Thus, by a similar argument as in Proposition \ref{polar}, 
$$
\sigma_{\a^p}(U)=\sigma_{\b(\h)}(U).
$$
It follows that $\log(U)\in\a^p_{ah}\cap\w$. One obtains local charts around other elements of $\u_{\a^p}$ translating this chart around $1$, by means of the left action of this group on itself.

That the group operations (multiplication and taking adjoint) are smooth is clear: these operations are smooth in the whole Banach algebra $\a^p$.  
\end{proof}

Note that if $P\in\p^p$, then $P^\perp:=1-P$ also belongs to $\p^p$. 
Let $P\in\p^p$ and $A\in\a^p$, denote by 
$$
{\bf S}_{P,A}=AP+(1-A)P^\perp\in\a^p .
$$
\begin{lem}
There exists an open neighbourhood $\w_P=\{A\in\a^p_h: \|A-P\|_{\infty,p}<r_P\}$ (for a given $r_P>0$) of $P$ in $\a^p$,  such that if $A\in\w_P$ then ${\bf S}_{P,A}\in \G_{\a^p}$.
\end{lem}
\begin{proof}
If $A=P$, then ${\bf S}_{P,P}=1$, so for $A\in \a^p_h$ close enough to $P$, ${\bf S}_{P,A}$ remains in $\G_{\a^p}$, which is open in $\a^p$.
\end{proof}
If $A=Q\in\p^p$, then ${\bf S}_{P,Q}$ is a standard element used to intertwine $P$ and $Q$: clearly
$$
{\bf S}_{P,Q}P=PQ=Q{\bf S}_{P,Q}.
$$
If, additionally, $Q$ belongs to $\w_P$, then ${\bf S}_{P,Q}$ is invertible and 
$$
Q={\bf  S}_{P,Q}P{\bf  S}^{-1}_{P,Q}.
$$
It is also a standard procedure to use ${\bf U}_{P,Q}$, the unitary part in the polar decomposition ${\bf S}_{P,Q}={\bf U}_{P,Q}|{\bf S}_{P,Q}|$, to obtain a unitary that intertwines
$$
Q={\bf U}_{P,Q}P{\bf U }^*_{P,Q}.
$$
Thus, a continuous map is defined
\begin{equation}\label{seccion local}
\mu_P:\p^p\cap\w_P\to \u_{\a^p} \ , \ \ \mu(Q)={\bf U}_{P,Q}.
\end{equation}
\begin{rem}
Let us denote by 
$$
\pi_P:\u_{\a^p}\to \p^p , \ \ \pi_P(U)=UPU^*.
$$
Clearly $\pi_P$ is a continuous map, whose image is the unitary orbit $\o_P=\{UPU^*:U\in\u_{\a^p}\}$ of $P$. The map $\mu_P$ of (\ref{seccion local}) is a continuous local cross section for $\pi_P$:
$$
\pi_P(\mu_P(Q))=Q, \hbox{ for } Q\in\p^p\cap\w_P.
$$
By translating this section using the left action of $\u_{\a^p}$ on itself, one obtains local cross sections on neighbourhoods of any point $P'$ in  $\p^p$.

By general topological considerations, this fact implies that the orbit $\o_P$ is open and closed in $\p^p$, and therefore a union of connected components.
This union is necessarily discrete, because as we have seen,
close projection belong to the same orbit.
\end{rem}

Let us show that $\p^p$ is a complemented submanifold of $\a^p_h$ (and since $\a^p$ is a (real) complemented subspace of $\a^p$, $\p^p$ is also a complemented submanifold of $\a^p$). Also, the same argument shows that the map $\pi_P:\u_{\a^p}\to \o_P$ is a submersion. To prove this fact we shall use the next general result, which can be found in \cite{rae}, and is a consequence of the implicit function theorem in Banach spaces.

\begin{lem}\label{raeburn}
Let $G$ be a Banach-Lie group acting smoothly on a Banach space $X$. For a fixed
$x_0\in X$, denote by $\pi_{x_0}:G\to X$ the smooth map $\pi_{x_0}(g)=g\cdot
x_0$. Suppose that
\begin{enumerate}
\item
$\pi_{x_0}$ is an open mapping,  regarded as a map from $G$ onto the orbit
$\{g\cdot x_0: g\in G\}$ of $x_0$ (with the relative topology of $X$).
\item
The differential $d(\pi_{x_0})_1:(TG)_1\to X$ splits: its null space and range are
closed complemented subspaces.
\end{enumerate}
Then the orbit $\{g\cdot x_0: g\in G\}$ is a smooth submanifold of $X$, and the
map
$$
\pi_{x_0}:G\to \{g\cdot x_0: g\in G\}
$$ 
is a smooth submersion.
\end{lem}
Here smooth means $C^\infty$.

\begin{teo}
$\p^p$ is a $C^\infty$ complemented submanifold of $\a^p_h$, and for any $P\in\p^p$ the map 
$$
\pi_P:\u_{\a^p}\to \o_P
$$
is a $C^\infty$ submersion.
\end{teo}
\begin{proof}
Let us use Lemma \ref{raeburn} in our context, namely: $X=\a^p_h$, $G=\u_{\a^p}$ and $x_0=P$. Clearly, $\pi_P$ is an open mapping, because it has local continuous cross sections. Let us denote by $\Pi=(d\pi_P)_1:\a^p_{ah}\to \a^p_h$ and prove that it splits. The cross section $\mu_P$ can be extended to a map defined in $\w_P$ (which is open in $\a^p_h$): let 
$$
\tilde{\mu}_P:\w_P\to \u_{\a^p} , \tilde{\mu}_p(A)={\bf S}_{P,A}|{\bf S}_{P,A}|^{-1},
$$
i.e. $\tilde{\mu}_P(A)$ is the unitary part in the polar decomposition of the invertible element ${\bf S}_{P,A}$. Clearly $\tilde{\mu}_P$ is a $C^\infty$ extension of $\mu_P$, defined on an open set in $\a^p_h$. Let us denote by 
$\Sigma=(d\tilde{\mu}_P)_P$. Clearly $\Sigma:\a^p_H\to\a^p_{ah}$. Note that, since $\pi_P$ takes values in $\p^p$, on a neighbourhood of $1\in\u_{\a^p}$ one has
$$
\pi_P\ \tilde{\mu}_P\ \pi_P=\pi_P\ \mu_P\ \pi_p= \pi_P.
$$
Differentiating this identity at $1$ one gets
$$
\Pi\ \Sigma\ \Pi=\Pi \ \hbox{ in } \a^p_{ah}.
$$
This implies that both $\Pi\Sigma$ and $\Sigma\Pi$ are idempotent operators, acting  $\a^p_h$ and $\a^p_{ah}$, respectively.
Thus $R(\Pi\Sigma)\subset\a^p_{h}$ is complemented, and note that 
$$
R(\Pi\Sigma)\subset R(\Pi)=R(\Pi\Sigma\Pi)\subset R(\Pi\Sigma),
$$
i.e. $R(\Pi\Sigma)=R(\Pi)$. Similarly, $N(\Sigma\Pi)=N(\Pi)$ is complemented in $\a^p_{ah}$, i.e. $\Pi=(d\pi_P)_1$ splits.
\end{proof}

\begin{rem}
In particular, the tangent space $(T\p^p)_P$ at $P\in p^p$ is given by
$$
(T\p^p)_P=\{[X,P]: X\in\a^p\}.
$$
\end{rem}
\section{Spectral picture of projections in $\p^p$}
If $P\in\p^p$ has matrix (in terms of $\h=\h_+\oplus\h_-$)
$$
P=\left(\begin{array}{cc} x & a \\ a^* & y \end{array}\right),
$$
then the fact that $P^2=P\ge 0$ implies that $x,y\ge 0$,
 $x-x^2=aa^*$, $y-y^2=a^*a$, and $xa+ay=a$. 
Moreover, $a\in\b_p(\h_-,\h_+)$. Let us state the following elementary consequences of these relations:
\begin{lem}
With the above notations, one has that $\|a\|\le 1/2$, and the eigenvalues of $x$ and $y$ are of the form
$$
t^+=\frac12+ \sqrt{\frac14 -s^2} \ \hbox{ or } \  t^-=\frac12- \sqrt{\frac14 -s^2}
$$
where $s\le \frac12$ is a  a singular value of $a$. One or both $t_+,t_-$ may occur.
\end{lem}
\begin{proof}
Clearly $\|x\|\le 1$ and $\|y\|\le 1$. Then
$$
\|a\|^2=\|aa^*\|=\|x-x^*\|=\sup\{t-t^2: t\in\sigma(x)\}\le \sup\{t-t^2: t\in[0,1]\}=\frac14.
$$ 
Again, $x-x^2=aa^*$, and the fact that $aa^*$ is compact,  imply that the elements $t$  in the  spectrum of $x$ which are neither $0$ nor $1$ (which correspond to the spectral value $0$ for $aa^*$) are (finite or countable many) eigenvalues.  Moreover, if $t\ne 0,1$ is an eigenvalue of $x$, then
$$
t-t^2=s^2,
$$
for $s$ a singular value of $a$. Then either $t=t^+=\frac12+ \sqrt{\frac14 -s^2}$ or $t=t^-=\frac12- \sqrt{\frac14 -s^2}$. The same facts hold for $y$.
\end{proof}
Note that the biggest singular value $s=\frac12$ of $a$ corresponds to $t^+=t^-=\frac12$. 

The next result, which was proven for compact commutators in
\cite{conmcomp},
holds also in this context, and clarifies the relation between the (multiplicities of) the eigenvalues of $x$ and $y$. We include the proof because it is elementary and straightforward.
\begin{lem}\label{simetriaespectros}
If $\lambda \ne 0,1$ is an eigenvalue of $y$, then $1-\lambda$ is an eigenvalue of $x$, and the operator 
$a|_{N(y-\lambda 1_{\h_-})}$ maps $N(y-\lambda 1_{\h_-})$ isomorphically onto $N(x-(1-\lambda) 1_{\h_+})$. Thus in particular, these eigenvalues have the same multiplicity. Moreover,
$$
aP_{N(y-\lambda 1_{\h_-})}=P_{N(x-(1-\lambda) 1_{\h_-})}a.
$$
\end{lem}

\begin{proof} 
Let $\xi\in\h$, $\xi\ne 0$, such that $y\xi=\lambda \xi$ (with $\lambda \ne 0,1$). Then,  using the relation $a=xa+ay$,  one has
$$
a\xi=xa\xi+ay\xi=xa\xi+\lambda a\xi, \ \hbox{ i.e. } xa\xi=(1-\lambda)a\xi.
$$
Also note that
$$
N(a)=N(a^*a)=N(y-y^2)=N(y)\oplus N(y-1_{\h_-}),
$$ 
and thus $a\xi\ne 0$ is an eigenvector for $x$, with eigenvalue $1-\lambda$, and the map $a|_{N(y-\lambda 1_{\h_-})}$ is injective from $N(y-\lambda 1_{\h_-})$ to $N(x-(1-\lambda) 1_{\h_+})$. Therefore 
$$
\dim(N(y-\lambda 1_{\h_-}))\le \dim(N(x-(1-\lambda) 1_{\h_+}).
$$
By a symmetric argument, using $a^*$ (and the relation $ya^*+a^*x=a^*$), one obtains equality.

Pick now an arbitrary $\xi\in\h_-$, $\xi=\xi_1+\xi_2$, with $\xi_1\in\ N(y-\lambda 1_{\h_-})$ and $\xi_2\perp N(y-\lambda 1_{\h_-})$.
Then 
$$
aP_{N(y-\lambda 1_{\h_-})}\xi=a\xi_1.
$$
On the other hand
$$
P_{N(x-(1-\lambda) 1_{\h_+})}a\xi_1=a\xi_1,
$$
by the fact proven above. Let us see that $P_{N(x-(1-\lambda)1_{\h_+})}a\xi_2=0$, which would prove our claim. Since $\xi_2\perp N(y-\lambda 1_{\h_-})$, $\xi_2=\sum_{l\ge 2}\eta_l+\eta_0+\eta_1$, where $\eta_l$, $l\ge 2$, are eigenvectors of $y$ corresponding to eigenvalues $\lambda_l$ different from $0$, $1$ and $\lambda$, $\eta_0\in N(y)$, $\eta_1\in N(y-1_{\h_-})$ (where these two latter may be trivial).
Note  then that  $\eta_0,\eta_1\in N(a)$, and thus
$$
a\xi_2=\sum_{l\ge 2}a\eta_l,
$$
where the (non nil) vectors $a\eta_l$ are eigenvectors of $x$ corresponding to eigenvalues $1-\lambda_l$, different from $0,1$ and $1-\lambda$.   Thus $P_{N(x-(1-\lambda) 1_{\h_+})}a\xi_2=0$.
\end{proof}
\begin{rem}

\noindent

\begin{enumerate}
\item
In the notation above, this  Lemma says that to $t^+$ of $x$ corresponds $t^-=1-t^+$ of $y$, and vice versa, with the same multiplicity.
\item
If $a$ has infinite rank,  $s=s_n$ form a sequence in $\ell^p$. If there are infinitely many $t^-$, then they form a sequence in $\ell^{p/2}$. If there are infinitely many $t^+$, they form a sequence $t^+_n$ such that $1-t^+_n$ belongs to $\ell^{p/2}$. Indeed, note that near the origin,   $f(s)=\frac12-\sqrt{\frac14-s^2}=s^2+o(s^4)$ 
\end{enumerate}
\end{rem}
We will characterize the connected components of $\p^p$. To this effect, it shall be useful and clarifying to consider the $*$-homomorphism
$$
\pi:\b(\h)\to\b(\h)/\k(\h)
$$
onto the Calkin algebra $\b(\h)/\k(\h)$. Note that $\pi(E_+), \pi(E_-)$ are non trivial projections with $\pi(E_+)+ \pi(E_-)=1$. Thus, elements in $\b(\h)/\k(\h)$ can be written as $2\times 2$  matrices in terms of this sum. Let us write
$$
\pi(E_+)=\left( \begin{array}{cc} 1 & 0 \\ 0 & 0 \end{array} \right) \ \hbox{ and } \  \pi(E_-)=\left( \begin{array}{cc} 0 & 0 \\ 0 & 1 \end{array} \right).
$$
Then, if $P\in\p^p$, it follows that 
$$
\pi(P)=\left( \begin{array}{cc} e & 0 \\ 0 & f \end{array} \right),
$$
where $e,f$ are projections in $\b(\h)/\k(\h)$ with $e\le \pi(E_+)$ and $f\le \pi(E_-)$.
\begin{lem}\label{limitante}
Let $P,Q\in\p^p$ in the same connected component, say
$$
\pi(P)=\left( \begin{array}{cc} e & 0 \\ 0 & f \end{array} \right)  \ \hbox{ and } \ 
\pi(Q)=\left( \begin{array}{cc} e' & 0 \\ 0 & f' \end{array} \right).
$$
Then there exists a curve $\gamma(t)=\left( \begin{array}{cc} e(t) & 0 \\ 0 & f(t) \end{array} \right)$
of projections in $\b(\h)/\k(\h)$ such that $\gamma(0)=\pi(P)$ and $\gamma(1)=\pi(Q)$.
\end{lem}
\begin{proof}
Let $P(t)$ be a continuous path in $\p^p$ with $P(0)=P$ and $P(1)=Q$. Note that in particular $P(t)$ is continuous in the norm topology of $\b(\h)$, and clearly, $\pi(P(t))$ has diagonal matrix with respect to $\pi(E_+)+\pi(E_-)=1$.
\end{proof}
Recall that there are three classes of projections in $\b(\h)/\k(\h)$ modulo unitary equivalence: $0$, $1$ and $p\ne 0,1$. Also, two projections are connected by a continuous path of projections if and only if they are unitarily equivalent.

\begin{rem}\label{tipos} 
The projections in $\p^p$ can be classified in the following nine types:
\begin{enumerate}
\item
$P$ belongs to $\mathbb{D}_1$ if $\pi(P)=\left( \begin{array}{cc} 0 & 0 \\ 0 & 0 \end{array} \right);$
\item
$P$ belongs to $\mathbb{D}_2$ if $\pi(P)=\left( \begin{array}{cc} 1 & 0 \\ 0 & 1 \end{array} \right);$
\item
$P$ belongs to $\mathbb{D}_3$ if $\pi(P)=\left( \begin{array}{cc} 1 & 0 \\ 0 & 0 \end{array} \right);$
\item
$P$ belongs to $\mathbb{D}_4$ if $\pi(P)=\left( \begin{array}{cc} 0 & 0 \\ 0 & 1 \end{array} \right);$
\item
$P$ belongs to $\mathbb{E}_1$ if $\pi(P)=\left( \begin{array}{cc} e & 0 \\ 0 & 0 \end{array} \right),$
where $e\ne 0,1$;
\item
$P$ belongs to $\mathbb{E}_2$ if $\pi(P)=\left( \begin{array}{cc} 0 & 0 \\ 0 & f \end{array} \right),$
where $f\ne 0,1$;
\item
$P$ belongs to $\mathbb{E}_3$ if $\pi(P)=\left( \begin{array}{cc} e & 0 \\ 0 & 1 \end{array} \right)$
where $e\ne 0,1$;
\item
$P$ belongs to $\mathbb{E}_4$ if $\pi(P)=\left( \begin{array}{cc} 1 & 0 \\ 0 & f \end{array} \right);$
where $f\ne 0,1$; and
\item
$P$ belongs to $\mathbb{E}_5$ if $\pi(P)=\left( \begin{array}{cc} e & 0 \\ 0 & f \end{array} \right);$
where $e, f\ne 0,1$.
\end{enumerate}
\end{rem}

We call the classes $\mathbb{D}_i$ {\it discrete} and $\mathbb{E}_j$ {\it essential}. 

Summarizing, one has the following expression for an arbitrary projection in $\p^p$. We make  here a  slight change  of notation. Without loss of generality, we assume that $\h=\ele\times\ele$, $\h_+=\ele\times 0$ and $\h_-=0\times\ele$
\begin{teo}\label{forma general}
If $P\in\p^p$, then
$$ P=
\left(\begin{array}{ccc}
\d\sum_n \alpha_n  P_n+ \d\sum_m \beta_m Q_m+E_1
& | &
\d\sum_k \lambda_k \xi_k\otimes \xi'_k + \d\sum_l \mu_l\eta_l\otimes \eta'_l \\ 
 ---------- & &  ----------\\
\d\sum_k \lambda_k \xi'_k\otimes \xi_k + \d\sum_l \mu_l\eta'_l\otimes \eta_l
& | & 
\d\sum_n (1-\alpha_n)P'_n + \d\sum_m(1-\beta_m)Q'_m +E'_1
\end{array}\right),$$
where
\begin{itemize}
\item
The spectrum of $x$ (in $\b(\ele)$) consists of two strictly monotone (eventually finite) sequences $\alpha_n, \beta_m$, such that $\frac12>\alpha_n\to 0$, $\frac12\le \beta_m\to 1$, plus, eventually, $0$ and $1$, which may or may not be eigenvalues. The spectrum of $y$ consists of $1-\alpha_n$, $1-\beta_m$  and eventually $0$ and $1$  (with similar considerations).
\item
$r(P_n)=r(P'_n)$, $r(Q_m)=r(Q'_m)$. These multiplicities are finite.
\item
$r(P_m)\alpha_m$ and $1-r(Q_m)\beta_m$ belong to $\ell^{p/2}$; $\lambda_k=\sqrt{\alpha_k-\alpha_k^2}$ and $\mu_l=\sqrt{\beta_l-\beta_l^2}$.
\item
$E_1$ and $E'_1$ denote the spectral projections of $x$ and $y$, respectively, corresponding to the spectral value $1$. They can be nil, finite or infinite, and are unrelated.
\item
$$
\{\xi_k:k\ge 1\}, \{\xi'_k:k\ge 1\}, \{\eta_l:l\ge 1\}  \hbox{ and } \{\eta'_l:l\ge 1\}
$$ 
are orthonormal systems which span, respectively
$$
\oplus_{n\ge 1} R(P_n), \oplus_{n\ge 1} R(P'_n), \oplus_{m\ge 1} R(Q_m)
 \hbox{ and } \oplus_{m\ge 1} R(Q'_m),
$$
and consists of eigenvectors of $x$ and $y$ in the following manner:
$$
x\xi_k=\alpha_{n(k)}\xi_k, \  x\eta_l=\beta_{m(l)}\eta_l, \  y\xi'_k=(1-\alpha_{n(k)})\xi'_k  \ \hbox{ and }  \ y\eta'_l=(1-\beta_{m(l)})\eta_l'.
$$
\end{itemize}
\end{teo}

\section{Halmos decomposition}
Given two projections, in this case $P$ and $E_+$, the space $\h$ can be decomposed in $5$ orthogonal subspaces which reduce $P$ and $E_+$, namely
$$
\h=\left(R(P)\cap\h_+\right)\oplus \left(N(P)\cap\h_-\right)\oplus\left(R(P)\cap\h_-\right)\left(N(P)\cap\h_+\right)\oplus \h_0,
$$
where $\h_0$, the orthogonal complement of the sum of the first $4$, is usually called the {\it generic part} of $P$ and $E_+$. In \cite{halmos}, Halmos proved that there is a unitary isomorphism between $\h_0$ and a product space $\ele\times\ele$, and a positive operator $\Gamma$ with trivial null space and $\|\Gamma\|\le \pi/2$,  acting in $\ele$, such that the reductions  $E^0_+$ and $P_0$ of   $E_+$ and $P$ to $\h_0$  are unitarily equivalent to (respectively)
$$
\left( \begin{array}{cc} 1 & 0 \\ 0 & 0 \end{array}\right) \ \hbox{ and } \ \left( \begin{array}{cc} C^2 & CS \\ CS & S^2 \end{array}\right),
$$
where $C=\cos(\Gamma)$ and $S=\sin(\Gamma)$. It can be shown that $P_0$ and $E_+^0$ are unitarily equivalent:
$$
e^{iX}E_+^0e^{-iX}=P_0,
$$
where $X=X_{E_+^0,P_0}=\left( \begin{array}{cc} 0 & -i\Gamma \\ i\Gamma & 0 \end{array}\right)$.

In this decomposition of $\h$, the commutator has the form
$$
[E_+,P]=0 \oplus 0 \oplus 0 \oplus 0 \oplus \left( \begin{array}{cc} 0 & CS \\ -CS & 0 \end{array}\right)
$$
Therefore:
\begin{prop}
$P\in\p^p$ if and only if $CS\in\b_p(\ele)$. Moreover, this means the spectrum of $\Gamma$ is of the form $\{\gamma_{n}^+: n\ge 1\}\cup\{\gamma_k^-: k\ge 1\}$, where $\pi/4\le \gamma_n^+<\pi/2$ and $0<\gamma_k^-<\pi/4$ are  strictly monotone (eventually finite) sequences, 
$$
\Gamma=\sum_{k\ge 1} \gamma_k^- G_k^- +\sum_{n\ge 1} \gamma_n^+ G_n^+,
$$
for $G_n^+, G_k^-$ mutually orthogonal  projections in $\ele$, of ranks $r(G^+_n)=r_n^+<\infty$ and $r(G^-_k)=r_k^-<\infty$,
 and
$$
\{r_k^-\gamma_k^-\} , \ \{\pi/2-r_n^+\gamma_k^+\}\in\ell^p.
$$
\end{prop}
\begin{proof}
The first assertion is clear. Note that $CS=\cos(\Gamma)\sin(\Gamma)=\frac12\sin(2\Gamma)$. Thus, the facts that $\sin(2\Gamma)\in\b_p(\ele)$ and $0\le 2\Gamma\le \pi$, means that the spectrum of $\Gamma$ consists of eigenvalues which accumulate only (eventually) at $0$ and $\pi$. Since $C$ and $S$ have trivial null spaces, neither $0$ nor $\pi/2$ are eigenvalues of $\Gamma$.
\end{proof}

\subsection{Examples}

\begin{enumerate}
\item
Let $\h=L^2(\mathbb{T})$, $\mathbb{T}=\{z\in\mathbb{C}: |z|=1\}$ with normalized Lebesgue measure, and $\h_+=H^2(\mathbb{D})$ the Hardy space of the disk. Let $\varphi:\mathbb{T}\to\mathbb{C}$ be non vanishing and $C^1$. Then the projection $P_{\varphi H^2(\mathbb{D})}$ onto $\varphi H^2(\mathbb{D})$ belongs to the restricted Grassmannian given by the subspace $\h^2(\mathbb{D})$ (see \cite{segal}), with $[P_+,P_{\varphi H^2(\mathbb{D})}]\in\b_1(L^2(\mathbb{T}))$. Thus $P_{\varphi H^2(\mathbb{D})}\in\mathbb{D}_3$. 
\item
Let $\h=L^2(\mathbb{R}^n)$ (with Lebesgue measure),
$\Omega\subset\mathbb{R}^n$ a measurable set with
$|\Omega|<\infty$ and $\h_+=L^2(\Omega)$,
regarded a the subspace of $L^2(\mathbb{R}^n)$ of
classes of functions with essential support contained
in $\Omega$. Denote by
$F:L^2(\mathbb{R}^n)\to L^2(\mathbb{R}^n)$
the Fourier-Plancherel transform. Put
$P=FE_+F^{-1}$,
which projects onto functions whose Fourier transform
is supported in $\Omega$.
It is known  (see for instance \cite{follandsitaram}) that 
$E_+P\in\b_1(L^2(\mathbb{R}^n))$. Thus, in particular,
$$
[E_+,P]=E_+P-PE_+=E_+P-(E_+P)^*\in\b_1(L^2(\mathbb{R}^n)).
$$
Moreover, Lenard proved in  \cite{lenard} that 
$$
R(E_+)\cap R(P)=R(E_+)\cap N(P)=N(E_+)\cap R(P)=\{0\},
$$
and that $N(E_+)\cap N(P)$ is infinite dimensional. Therefore, since $E_+PE_+$, $E_+PE_-$ and $E_-PE_+$ are compact, it follows that $\pi(P)$ is of the form
$$
\pi(P)=\left( \begin{array}{cc} 0 & 0 \\ 0 & p \end{array}\right),
$$ 
with $p\ne 0, 1$, because  $P$ is an infinite rank projection, with $N(P)$ infinite dimensional.  That is, $P\in\mathbb{E}_2$ .
\item
Let $B\in\b_p(\ele)$, and put $\h=\ele\times\ele$ and $\h_+=\ele\times 0$. Consider the idempotent (non orthogonal projection)
$$
E_B=E=\left( \begin{array}{cc} 1 & B \\ 0 & 0 \end{array} \right)
$$ 
with $R(E)=\h_+$. Consider $P_{R(E^*)}$ the orthogonal projection onto $R(E^*)$. It is known (see for instance \cite{ando}), that if $Q$ is an idempotent operator, then $P_{R(Q)}=Q(Q+Q^*-1)^{-1}$. In our case, note that 
$(E^*+E-1)^2=\left(\begin{array}{cc} 1+BB^* & 0 \\ 0 & 1+ B^*B \end{array}\right)$, so that
$$
P=P_{R(E^*)}=E^*(E^*+E-1)^{-1}=E^*(E^*+E-1)(E^*+E-1)^{-2}
$$
$$
=
\left(\begin{array}{cc} (1+BB^*)^{-1} & B(1+B^*B)^{-1} \\ B^*(1+BB^*)^{-1} & B^*B(1+B^*B)^{-1} \end{array}\right).
$$
Note that the $1,2$ entry $B(1+B^*B)^{-1}$ belongs to $\b_p(\h)$, with singular values which have the same asymptotic behaviour as those of $B$. Clearly, $\pi(P)$ is of the form $\left(\begin{array}{cc} 1 & 0 \\ 0 & 0 \end{array}\right)$, i.e. $P\in\mathbb{D}_3$. The index (of the  $1,1$ entry) of $P$ is $0$.
\end{enumerate}

\section{The classes $\mathbb{D}_3$ and $\mathbb{D}_4$}

Let us show that $\mathbb{D}_3$ coincides with the $p$-restricted Grassmannian induced by $E_+$. Recall that (see for instance \cite{segal})  the \textit{$p$-restricted Grassmannian} $Gr^p_{res}(E_+)$ relative to $E_+$,  is the space projections $P$ in $\h$  such that
\begin{itemize}
\item $E_+|_{R(P)}:R(P)\to \h\in\b(R(P),\h)$
is a $p$-Fredholm operator (i.e., there exist $S\in \b(\h,R(P))$ such that $SE_+|_{R(P)}=1+M$ and $E_+|_{R(P)}S=1+N$, for $M\in\b_p(R(P))$, $N\in\b_p(\h)$), and
\item $E_-|_{R(P)}:R(P)\to \h\in\b_p(R(P),\h)$.
\end{itemize}
The components of the restricted Grassmannian are parametrized by $k\in\mathbb Z$, where $k$ is the index of the operator $E_+|_{R(P)}:R(P)\to \h\in\b(R(P),\h)$,
$$
Gr_{res,k}^p(E_+)=\{P\in Gr^p_{res}: ind(E_+|_{R(P)}:R(P)\to \h)=k\}.
$$
In particular, note that $E_+\in Gr_{res,0}^p(E_+)$.

The coincidence of the $p$-restricted Grassmannian of $\h_+$ and $\mathbb{D}_3$ follows from this result:
\begin{teo}
Denote by $\o(E_+)$ the unitary orbit of $E_+$ under the action of $\u_{\a^p}$,
$\o(E_+)=\{UE_+U^*: U\in\u_{\a^p}\}$. Then
$$
\o(E_+)=\{P\in\p^p: P-E_+\in\b_p(\h)\}=\{P\in\p: P-E_+\in\b_p(\h)\}.
$$
\end{teo}
\begin{proof}
If $P=UE_+U^*$ for some $U\in\u_{\a^p}$, then 
$$
P-E_+=UE_+U^*-E_+=(UE_+-E_+U)U^*=[U,E_+]U^*\in\b_p(\h):
$$
Clearly
$\{P\in\p^p: P-E_+\in\b_p(\h)\}\subset\{P\in\p: P-E_+\in\b_p(\h)\}$.
Suppose that $P\in\p$ such that $P-E_+\in\b_p(\h)$.
It is easy to see that
$$N(P-E_+)=R(P)\cap \h_+\oplus N(P)\cap \h_-.$$
Also it is clear that both summands reduce $P$ and $E_+$.
Then $\h'=N(P-E_+)^\perp$ reduces $P$ and $E_+$.
Denote by $P'$ and $E_+'$ the reductions.
It is straightforward that
$[P',E_+']\in\b_p(\h')$.
Since $P'-E_+'$
is selfadjoint and has trivial null space,
if one performs the polar decomposition 
$$
P'-E_+'=V'|P'-E_+'|,
$$
the isometric part $V'$ is a symmetry (a selfadjoint unitary) in $\h'$. Also, the fact that $S'=P'-E_+'$ satisfies $S'E_+'=E_+'S'$ implies that $V'$ intertwines $E_+'$ and $P'$: $V'E_+'V'=P'$. Then, it also follows that $V'\in\{X'\in\b(\h'): [X',E_+']\in\b_p(\h')$, the algebra $\a^p$ in $\h'$ corresponding to the reduced projection $E_+'$. Indeed,
$$
V'E_+'-E_+'V'=(V'E_+'V'-E_+')V'\in\b_p(\h').
$$
Consider now the unitary operator (in fact symmetry) $V$ of $\h$, which is given in terms of the decomposition $\h=\h'\oplus (R(P)\cap \h_+) \oplus (N(P)\cap \h_-)$ is given by 
$$
V'\oplus 1 \oplus 1.
$$
Note that in this same decomposition, $P$ and $E_+$ are given by
$$
P=P'\oplus 1 \oplus 0 \ \hbox{ and } \ E_+=E_+'\oplus 1 \oplus 0.
$$
Then 
$$
[V,E_+]=(V'E_+'-E_+'V')\oplus 0 \oplus 0 \in\b_p(\h),
$$
and 
$$
VE_+V=(V'E_+'V')\oplus 1 \oplus 0 =P.
$$ 
\end{proof}
\begin{coro}
$\mathbb{D}_3=\o(E_+)$
\end{coro}
\begin{proof}
$P\in\o(E_+)$, if and only if $P-E_+\in\b_p(\h)$, and then $\pi(P-E_+)=0$,
i.e. $\pi(P)=\pi(E_+)$. Conversely, in $\pi(P)=\pi(E_+)$, then $P-E_+$ is compact.
Using a suitable unitary isomorphism we may suppose (as in Theorem \ref{forma general})   $\h=\ele\times\ele$ and $\h_+=\ele\times 0$, and use the spectral picture of $P\in\p^p$. Note that the assumption that $P-E_+$ is compact implies that $x-1$ and $y$ are compact. Thus, following the notation of Theorem \ref{forma general}, one has that there are finitely many $\alpha_n$ and that $1-\beta_n$ is a sequence in $\ell^{p/2}$. It follows that $P-E_+\in\b_p(\h)$.
\end{proof}

\begin{rem}
Note that in particular, this facts imply that  $\mathbb{D}_3\in E_++\b_p(\h)$, i.e. $\mathbb{D}_3$ is contained in the affine space obtained as a translation of $\b_p(\h)$. Thus, the tangent spaces belong naturally inside $\b_p(\h)$. We shall profit from this condition, in order to endow the manifold $\mathbb{D}_3$ with the natural Finsler metric, which consists in considering the $p$-norm at every tangent space.
\end{rem}
\begin{rem}
In a similar fashion (or using the symmetry $P\mapsto P^\perp$), one proves that 
$$
\mathbb{D}_4=\o(E_-)=\{P\in\p: P-E_-\in\b_p(\h)\},
$$
which coincides with the $p$-restricted Grassmannian $G_{res}^p(E_-)$ induced by $E-$. Similarly, one can consider the Finsler $p$-norm structure in $\mathbb{D}_4$.
\end{rem}

In general, if $P$ and $Q$ are projections, $\|P-Q\|\le 1$. 
\begin{prop}
If $P\in\p^p$ satisfies that $\|P-E_+\|<1$, then $P\in\mathbb{D}_3$. Similarly, if $\|P-E_-\|<1$, then $P\in\mathbb{D}_4$.
\end{prop}
\begin{proof}
Recall Halmos decomposition. Clearly, $\|P-E_+\|<1$ implies that $R(P)\cap\h_-=N(P)\cap\h_+=\{0\}$. Indeed, a unit vector $\xi\in R(P)\cap \h_-$ satisfies $\|(P-E_+)\xi\|=\|\xi\|=1$, and thus $R(P)\cap \h_-=\{0\}$; and similarly for the other intersection. 
Note that
$$
P-E_+=0\oplus 0\oplus 0\oplus 0\oplus \left(\begin{array}{cc} -S^2 & CS \\ CS & S^2 \end{array}\right)
$$
and therefore $(P-E_+)^2=\left(\begin{array}{cc} S^2 & 0 \\ 0 & S^2\end{array}\right)$.
Then $\|P-E_+\|<1$ implies that the spectrum of $\Gamma$ cannot accumulate at $\pi/2$, and therefore only accumulates only at the origin. Thus, analysing
the spectral picture of $P=\left(\begin{array}{cc} C^2 & CS \\ CS & S^2\end{array}\right) $ according to Theorem \ref{forma general}, it is clear that the eigenvalues of $x=C^2$ accumulate only at $1$ and the eigenvalues of $y=S^2$  accumulate only at the origin.

The proof for the case $\|P-E_-\|<1$ is analogous.
\end{proof}
\begin{coro}
If $P\in\mathbb{E}_j$, $1\le j\le 5$, then $\|P-E_+\|=\|P-E_-\|=1$.
\end{coro}

\subsection{Finsler metric in the discrete classes}

In this subsection we examine the relationship between the $p$-Finsler metric of $\mathbb{D}_3$ and $\mathbb{D}_4$, i.e., the metric which arises when endowing each tangent space of these manifolds with the $p$-norm,  with the ambient metric given by the norm $\| \ \ \|_{\infty,p}$ of $\a^p$.

Let us reason with $\mathbb{D}_3$, the same facts hold analogously for $\mathbb{D}_4$.  

Let us first compare the Finsler metric with the metric induced by the $p$ norm of the affine space $E_++\b_p(\h)$ (if $X=E_++A, Y=E_++B\in E_++\b_p(\h)$, the $p$ distance $\|X-Y\|_p=\|A-B\|_p<\infty$ is defined).

In \cite{al} it was proven that if $P,Q$ lie in the same component $G_{res,k}^p(E_+)=\mathbb{D}_3$ of the $p$-restricted Grassmannian , then there exists a minimal geodesic of the form
$\delta(t)=e^{itX}P e^{-itX}$, with $X^*=X\in\b_p(\h)$ $P$-codiagonal and $\|X\|\le\pi/2$, such that $\delta(1)=Q$, so that
the geodesic distance $d_p(P,Q)=\|X\|_p$. We recall that $d_p$ is formally defined as
$$
d_p(P,Q)=\inf\{ \int_I \|\dot{\gamma}(t)\|_p\  d t : \gamma:I\to\p^p, \gamma \hbox{ is smooth with endpoints } P,Q\}.
$$
\begin{prop}
With the current notations,
if $P,Q$ lie in the same component of $\mathbb{D}_3$
(resp. $\mathbb{D}_4$),
then
$$\frac{2}{\pi} d_q(P,Q)\le \|P-Q\|_p\le d_p(P,Q).$$
\end{prop}

\begin{proof}
The inequality $\|P-Q\|\le d_p(P,Q)$ is clear: if one takes the infimum
among all smooth curves with values in $E_++\b_p(\h)_h$, which is an affine space, one obtains the norm distance $\|P-Q\|_p$.

Let $X=X^*\in\b_p(\h)$ be the exponent of the geodesic joining $P$ and $Q$: $\|X\|\le \pi/2$, $X$ is $P$-codiagonal, and $Q=e^{i X}Pe^{-iX}$. Note that
$$
\|P-Q\|_p=\|P-e^{iX}Pe^{-iX}\|_p=\|(Pe^{iX}-e^{iX}P)e^{-iX}\|_p=\|[P,e^{iX}]\|_p.
$$
Since $X$ is $P$-codiagonal, $P$ commutes the the even powers of $X$, and thus
$$
[P,X^{2n+1}]=PX^{2n}X-X^{2n}XP=X^{2n}(PX-XP)=X^{2n}[P,X].
$$
It follows that
$$
[P,e^{iX}]=[P, 1+iX-\frac{1}{2}X^2-\frac{i}{3 !}X^3+\frac{1}{4 !} X^4+\dots]=i\{[P,X]-\frac{1}{3 !}[P,X^3]+\frac{1}{5 !}[P,X^5]-\dots\}
$$
$$
=i\{1-\frac{1}{3 !}X^2+\frac{1}{5 !}X^4-\dots\}[P,X]=i \ \sinc (X)[P,X],
$$
where $\sinc$ denotes the {\it cardinal sine} function,
which is the entire function given by
$\sinc(t)=\frac{\sin(t)}{t}$ ($\sinc(0)=1$). It $|t|\le\pi/2$, this function verifies that
$$
\frac{2}{\pi}\le \sinc(t)\le 1.
$$ 
In particular, since $X$ is selfadjoint with spectrum in $[-\pi/2,\pi/2]$,
$S=\sinc(X)$ is an invertible operator, with $\|S^{-1}\|\le \frac{\pi}{2}$.
Therefore
$$
\|X\|_p=\|S^{-1}SX\|_p\le \|S^{-1}\|\|SX\|_p\le\frac{\pi}{2}\|SX\|_p,
$$
i.e.,
$$
\frac{2}{\pi}\|X\|_p\le \|SX\|_p=\|i\ \sinc(X)[P,X]\|_p=\|P-Q\|_p.
$$
\end{proof}

\begin{rem}
By the above remarks, if $P,Q\in\mathbb{D}_3$,
then $P-Q\in\b_p(\h)_h$.
Denote
$P-Q=A=
\left(\begin{array}{cc}
A_{11} & A_{12}\\
A_{12}^* & A_{22}
\end{array}\right)$.
Since $A_{ij}=E_{ij}A$ for appropriate elementary
(partial isometric) operators, it is clear that
$A_{ij}\in\b_p(\h)$.
Moreover
$\|A_{ij}\|_p=\|E_{ij}A\|_p\le\|E_{ij}\|\|A\|_p=\|A\|_p$.
Then
$$\begin{array}{ccl}
\|P-Q\|_{\infty,p}& = & \|A_{11}\|+\|A_{22}\|+\|A_{12}\|_p+\|A_{12}^*\|_p \\
& \leq & \|A_{11}\|_p+\|A_{22}\|_p+\|A_{12}\|_p+\|A_{12}^*\|_p \\
& \leq &  4\|P-Q\|_p.
\end{array}$$
However, these two metrics are not equivalent in $\mathbb{D}_3$ 
Indeed, fix an orthonormal basis $\{f_n\}$ for $\h_-$,
and  consider $P=E_++D$ and $Q=E_++F$, where $D,F\le E_-$,
project onto mutually orthogonal subspaces generated by finite
(disjoints) subsets of the basis $\{f_n\}$.
Then $\|P-Q\|_{\infty,p}=\|D-F\|=1$, whereas
$\|P-Q\|_p=(\rank(E)+\rank(F))^{1/p}$.
Since these ranks are arbitrary,
the metrics are non equivalent.

\end{rem}

\section{Connectedness of the essential classes}

In this section we prove our main result in the classes $\mathbb{E}_i, 1\le i\le 5$, namely, that each of these spaces is connected. The proof of this result is similar to the proof of the analogous result in \cite{conmcomp}. We shall sketch the argument, emphasizing only the necessary modifications.

Recall from Theorem \ref{forma general}, the form of a projection $P=\left( \begin{array}{cc} x & a \\ a^* & y \end{array}\right)\in\p^p$:
$$ P=
\left(\begin{array}{ccc}
\d\sum_n \alpha_n  P_n+ \d\sum_m \beta_m Q_m+E_1
& | &
\d\sum_k \lambda_k \xi_k\otimes \xi'_k + \d\sum_l \mu_l\eta_l\otimes \eta'_l \\ 
 ---------- & &  ----------\\
\d\sum_k \lambda_k \xi'_k\otimes \xi_k + \d\sum_l \mu_l\eta'_l\otimes \eta_l
& | & 
\d\sum_n (1-\alpha_n)P'_n + \d\sum_m(1-\beta_m)Q'_m +E'_1
\end{array}\right),$$
where the relevant facts we need now are that that $r(P_n)=r(P'_n)<\infty$, $r(Q_m)=r(Q'_m)<\infty$, and $\alpha_n$, $1-\beta_m$ belong to $\ell^{p/2}$.
 Consider the projection
$$
P_0=\left(\begin{array}{ccc}
\d\sum_n P_n+ \d\sum_m Q_m+E_1+N
& | &
0 \\ 
 ---------- & &  ----------\\
0
& | & 
\d\sum_n P'_n + \d\sum_m Q'_m +E'_1+N'
\end{array}\right),
$$
where $N$ and $N'$ are the projections onto the nullspaces of $x$ and $y$, respectively.

The first step of the argument is the following.
\begin{lem}
The operator $B=P+P_0-1$ is invertible in $\a^p$, and belongs to the connected component of the identity.
\end{lem}
\begin{proof}
In \cite{conmcomp}, Lemma 5.1, it was shown that $B$ is invertible,  that its commutator with $E_+$ is compact, and that its $1,1$ entry is invertible. Here, $[B,P_+]$ is the $P_+$-codiagonal matrix whose non nil entries are those of $P$, and therefore $B\in\a^p$. Clearly, it belongs to the component of zero index.
\end{proof}
The operator $B$ is selfadjoint, and satisfies $BP=P_0B$, i.e. $BPB^{-1}=P_0$.  Then the unitary part $V$ in the polar decomposition $B=V|B|$, satisfies $VPV^*=P_0$. Clearly, $V$ also belongs to the connected component of the identity in $\u_{\a^p}$.  It follows that $P$ and $P_0$ belong to the same connected component of $\p^p$.

The next step is to show that any pair of {\it diagonal} essential projections in the same class $\mathbb{E}_i$, are conjugate by an element in the connected component of the identity of $\u_{\a^p}$:

Let $F, G$ be two projections, which are diagonal with respect to $E_+$, both in the same essential class.
\begin{itemize}
\item
If $F, G\in\mathbb{E}_1$,  are of the form
$$
F=\left(\begin{array}{cc} P_+ & 0 \\ 0 & F_- \end{array}\right) \ , \ \ G=\left(\begin{array}{cc} P'_+ & 0 \\ 0 & G_- \end{array}\right)
$$
where $P_+, P'_+$ are projections of infinite rank and co-rank, and $F_-, G_-$ are of finite rank. One can show that $F$ and $G$ are unitarily equivalent to the projection
$$
\left(\begin{array}{cc} P_+ & 0 \\ 0 & 0 \end{array}\right),
$$ 
with a unitary in operator in $\a^p$, which belongs to the connected component of the identity. First note that with a unitary of the form $\left(\begin{array}{cc} U_+& 0 \\ 0 & 1 \end{array}\right)$, one can connect $G$ with $\left(\begin{array}{cc} P_+ & 0 \\ 0 & G_- \end{array}\right)$. That is, we may suppose $P'_+=P_+$. Next we construct the unitary operator $U$ given in the proof of Lemma 5.2 of \cite{conmcomp}, which is a finite rank perturbation of the identity, and therefore also in the connected component of the identity of the invertible group of $\a^p$. This unitary connects $F$ to $\left(\begin{array}{cc} P_+ & 0 \\ 0 & 0 \end{array}\right)$, and then  $\mathbb{E}_1$ is connected.
\item
The connectedness of $\mathbb{E}_3$  can be obtained noting that the mapping $P\mapsto P^\perp=1-P$ tansforms $\mathbb{E}_1$ into $\mathbb{E}_3$.
\item
The case of $\mathbb{E}_2$ is analogous to the case of $\mathbb{E}_1$, and therefore the case of $\mathbb{E}_4$ also follows.
\item
If $F,G\in\mathbb{E}_5$, they are of the form
$$
F=\left(\begin{array}{cc} F_+ & 0 \\ 0 & F_- \end{array}\right) \ , \ \ G=\left(\begin{array}{cc} G_+ & 0 \\ 0 & G_- \end{array}\right),
$$
with $F_\pm, G_\pm$ of infinite rank and co-rank. Clearly these projections are unitarily equivalent with a diagonal unitary matrix, which therefore belong to the connected component of the identity in the invertible group of $\a^p$.
\end{itemize}
Thus, we have the following:
\begin{teo}
The classes $\mathbb{E}_i$, $1\le i\le 5$ are connected. Each of these spaces is the orbit of a fixed (diagonal) projection in the corresponding class,  under the action of  the the unitary operators in the connected component of the identity in $\a^p$.
\end{teo}

\textsc{(Esteban Andruchow)}
Instituto de Ciencias,
Universidad Nacional de Gral. Sarmiento,
J.M. Gutierrez 1150,  (1613) Los Polvorines,
Argentina and Instituto Argentino de Matem\'atica,
`Alberto P. Calder\'on', CONICET, Saavedra 15 3er. piso,
(1083) Buenos Aires, Argentina.

\noi e-mail: {\sf eandruch@ungs.edu.ar}

{\sc (Mar\'\i a Eugenia Di Iorio y Lucero)} {Instituto Argentino
de Matem\'atica, `Alberto P. Calder\'on', CONICET,
Saavedra 15 3er. piso,
(1083) Buenos Aires, Argentina.}

%%%%%%%%%%%%%%%%%%%%%%%%%%%%%%%%%%%%%%%%%%%%%%%%%%%%%
%%%%%%%%%%%%%%%%%%%%%%%%%%%%%%%%%%%%%%%%%%%%%%%%%%%%%

\end{document}